\documentclass[12pt]{amsart}

\usepackage{amsmath, amsfonts, amssymb, amsthm, enumerate}
\usepackage{hyperref}
\usepackage{bm}
\usepackage{graphicx, pstricks}

\usepackage{breqn}

\makeatletter
\def\input@path{{./figures/}{./bibliography/}}
\makeatother

\newcommand{\eps}{{\varepsilon}}

\newcommand{\Schwartz}{{\mathcal{S}}}
\newcommand{\Sphere}{{\boldsymbol{S}}}
\newcommand{\Real}{{\boldsymbol{R}}}

\newcommand{\Hausdorff}{{\mathcal{H}}}

\newcommand{\One}{{\boldsymbol{1}}}

\newcommand{\mres}{\mathbin{\vrule height 1.6ex depth 0pt width 0.13ex\vrule
height 0.13ex depth 0pt width 1.3ex}}

\newtheorem{theorem}{Theorem}[section]
\newtheorem{lemma}{Lemma}[section]
\newtheorem{definition}{Definition}[section]
\newtheorem{corollary}{Corollary}[section]
\newtheorem{remark}{Remark}[section]
\newtheorem{proposition}{Proposition}[section]

\begin{document}
\title{Some properties of a Hilbertian norm for perimeter}
\author[F. Hern{\'a}ndez]{Felipe Hern{\'a}ndez}
\date{\today}
\email{felipeh@alum.mit.edu}

\begin{abstract}
    We investigate a relationship first described in~\cite{figalli2014recognize}
    between the perimeter of a set and a related fractional Sobolev norm.  In
    particular, we derive a new characterization of sets of finite perimeter, and
    demonstrate that the fractional Sobolev norm does not recover the $BV$
    norm but rather a certain quadratic integral.
\end{abstract}

\maketitle
In a recent paper of Jerison and Figalli~\cite{figalli2014recognize}, a
relationship is developed between the perimeter of a set and a fractional
Sobolev norm of its indicator function.  More precisely, letting $\gamma(x)$
denote the standard Gaussian in $\Real^n$, and
defining the scaled Gaussian $\gamma_\eps(x) = \eps^{-n}\gamma(x/\eps)$,
Jerison and Figalli showed that
\begin{equation}
    \limsup_{\eps\to 0^+} \frac{1}{|\log\eps|}\|\gamma_\eps\ast\One_E\|_{H^{1/2}}^2
    \simeq
    \liminf_{\eps\to 0^+} \frac{1}{|\log\eps|}\|\gamma_\eps\ast\One_E\|_{H^{1/2}}^2
    \simeq P(E).
    \label{H12-norm}
\end{equation}
Here $E$ is a set of finite perimeter and $\One_E$ is its indicator function.
The definition of the $H^{1/2}$ norm that we shall use for the paper is given
in Section~\ref{basic-ests-sec}.
The formula is remarkable for its quadratic scaling and for its apparent
Fourier-analytic nature.  The motivation for writing down this expression for
the perimeter actually came from a purely geometric question about characterizing
convex sets in terms of there marginals.

A very similar quantity has appeared in the literature before, in
the foundational work of Bourgain, Brezis, and Mironescu~\cite{bourgain2001another}.
There the one-dimensional expression
\[
    \lim_{\eps\to 0}\frac{1}{|\log\eps|}
    \int\int_{|x-y|>\eps}\frac{|f(x)-f(y)|^2}{|x-y|^2}\,dxdy
\]
appears as a remark referencing an earlier unpublished work of Mironescu
and Shafrir.  The motivation for studying this functional came from
studying $\Gamma$-limits of Ginzburg-Landau type functionals, as for example
in~\cite{kurzke2006boundary, kurzke2006nonlocal}.  More recently Poliakovsky
introduced a similar functional in connection to the $\Gamma$-limit of the
Aviles-Gila problem~\cite{poliakovskypreprint}.  Poliakovsky introduced the
notion of $BV^q$ spaces and showed that a certain nonlocal functional very
similar to that appearing in Equation~\eqref{H12-norm} captures the $L^q$ norm of the
jump set of a function.\footnote{The reader is invited to compare the results of this
    paper with Theorem 1.1 of~\cite{poliakovskypreprint}.  It is not clear to the
author if there are any direct implications in either direction, but the works
certainly seem related.}

We mention also several other works which relate perimeter and total variation
to nonlocal functionals.  A paper of Leoni and
Spector~\cite{leoni2011characterization} studies a fractional Sobolev expression
that recovers the total variation of a function.  The main difference is that
the integrand in their expression scales linearly in the function $u$, whereas
the formula~\eqref{H12-norm} scales quadratically.  Similarly, a recent paper
of Ambrosio, Bourgain, Brezis, and Figalli~\cite{ambrosio2016bmo} introduced a
very interesting $BMO$-type norm which recovers the perimeter of a set.
In~\cite{fusco2016formula} this norm was shown to also give the total variation
of functions in $SBV$.

\subsection{Summary of Results}
In this paper we prove several results with the purpose of trying to better
understand~\eqref{H12-norm}.  The first question we investigate is what happens to
sets that do not have finite perimeter.
The answer comes in two parts.  First, we show in Theorem~\ref{example-thm}
that there is a set $E\subset[0,1]^n$ for which the limit inferior
in~\eqref{H12-norm} vanishes.
We do this by showing it is possible to construct a set of infinite perimeter such
that, for an sequence $\eps_k\to 0$ the functions
$\gamma_{\eps_k}\ast \One_E$ are much smoother than $\gamma_{\eps_k}$.  The
construction is presented in Section~\ref{construction-section}.

The second part of the answer is that the limit superior does characterize
sets of finite perimeter.  This is proven in Section~\ref{characterize-sec}
using in a strong way the $L^2$ structure of the norm.  The proof relies on the
characterization of sets of finite perimeter provided in~\cite{ambrosio2016bmo}.

The next question this paper addresses is whether the limit in the
expression~\eqref{H12-norm} exists at all.  From the previous results it is clear
that the limit cannot always exist, as  it is possible for
\[
    0 = \liminf_{\eps\to 0}\frac{1}{|\log\eps|}\|\gamma_\eps\ast\One_E\|_{H^{1/2}}^2
    < \limsup_{\eps\to 0}\frac{1}{|\log\eps|}\|\gamma_\eps\ast\One_E\|_{H^{1/2}}^2
    = \infty
\]
when $E$ is a set of infinite perimeter.  We show however that this cannot
be the case when $E$ is a set of finite perimeter.  Indeed, we prove in
Theorem~\ref{sofp-thm} that
\[
    \lim_{\eps\to 0}\frac{1}{|\log\eps|}\|\gamma_\eps\ast\One_E\|_{H^{1/2}}^2
    = c_n P(E).
\]
The argument consists of a few local calculations to cover the smooth case
and is finished by classical structural results about sets of finite perimeter and a
typical covering argument.

The question of what happens to other functions in $BV$ remains open, and we
are only able to give partial results.  Our first result is that the leading order
in the divergence of the $H^{1/2}$ norm only recovers information about the jump
set.  That is, under the condition that the total variation of a function in
$u\in BV\cap L^\infty(\Real^n)$ vanishes on all sets of Hausdorff dimension $n-1$,
\[
    \lim_{\eps\to 0}\frac{1}{|\log\eps|}\|\gamma_\eps\ast u\|_{H^{1/2}}^2
    = 0.
\]
This can be interpreted as a trichotomy that it interesting in its own right:
if a function $u$ has sufficient energy at high frequencies (as measured by the
$H^{1/2}$ norm of $\gamma_\eps\ast u$), then it either fails to have bounded
variation, it is unbounded, or it has a jump discontinuity.  This result
is stated in a more quantitative form as Theorem~\ref{jump-bound-thm}
in Section~\ref{no-jumps-sec}.

We are able to use this result to completely resolve the situation in one
dimension.  In fact, Theorem~\ref{one-dim-thm} implies that for $u\in BV(\Real)$
\[
    \lim_{\eps\to 0}\frac{1}{|\log\eps|}\|\gamma_\eps\ast u\|_{H^{1/2}}^2
    = c_1 \sum_{x\in J_u} |u^+(x)-u^-(x)|^2.
\]
where $J_u$ is the jump part of $u$, as described for example
in~\cite[Section 5.9]{evans2015measure} or~\cite[Section 3.8]{ambrosio2000functions}.
This one-dimensional formula is suggestive of what should occur in higher
dimensions, but we are not able to prove anything definitive here.

\subsection{A note on organization}
Preliminary results and basic definitions appear in Section~\ref{basic-ests-sec}.
Each of the other sections can be read independently of each other, so the
reader should feel free to skip to the result that is most interesting to them.

\subsection{Acknowledgements}
I am grateful to David Jerison, who suggested this problem to me and helped me
to focus the investigation.  I would also like to thank Francesco Maggi for
providing the reference for Theorem~\ref{C1-approx}.  Daniel Spector pointed
out to me the very relevant citation~\cite{poliakovskypreprint}.  The author is supported
by the John and Fannie Hertz Foundation Fellowship.

\section{Setup and Basic Estimates}
\label{basic-ests-sec}
First we define more clearly the Sobolev norm that we shall use.  Given a
smooth function $u\in \Schwartz(\Real^n)$ with rapid decay, we define
\[
    \|u\|_{H^{1/2}}^2 = \int|\xi||\widehat{u}(\xi)|^2\,d\xi.
\]
We would like to study the expression
\[
    \frac{1}{|\log\eps|} \|\gamma_\eps \ast u\|_{H^{1/2}}^2
\]
for functions $u\in BV\cap L^\infty(\Real^n)$ by decomposing the contributions
from different scales.  To do this, we write
\[
    \|\gamma_\eps \ast u\|_{H^{1/2}}^2
    = \sum_{k=1}^{|\log\eps|} \eps^{-1}2^{-k} \|\varphi_{\eps2^k} \ast u\|_{L^2}^2
    + O(\|u\|_{L^2})
\]
with the convolution kernel $\varphi$ chosen such that $\widehat{\varphi}(\xi)$ is
nonnegative and
\begin{equation}
    |\widehat{\varphi}(\xi)|^2 = |\xi|\left(|\widehat{\gamma}(\xi)|^2
    - |\widehat{\gamma}(2\xi)|^2\right).
    \label{special-kernel-defn}
\end{equation}
Observe that $|\xi|^{-3/2}\widehat{\varphi}(\xi)$ is continuous and bounded, and that
$\widehat{\varphi}(\xi)$ is smooth outside the origin.  From this we deduce that $f$
has the decay
\[
    |\varphi(x)| \leq C (1+|x|)^{-\frac{3}{2}-n}.
\]
In particular, $|x|\varphi \in L^1(\Real^n)$.  In addition, $\varphi$ is smooth and
satisfies the cancellation condition $\int \varphi = 0$.  These three conditions
are sufficient for us to prove the following elementary estimates that we
will use throughout the paper.
\begin{lemma}
    Let $f$ be any smooth function such that $\int f = 0$ and
    $|x|f \in L^1(\Real^n)$.
    Then we have the global $L^\infty$ bound
    \begin{equation}
        \|f_r\ast u\|_{L^\infty} \leq C \|u\|_{L^\infty},
        \label{Linf-conv-leq}
    \end{equation}
    and the $BV$ bound
    \begin{equation}
        \frac{1}{r}\|f_r\ast u\|_{L^1} \leq C \|u\|_{BV}.
        \label{BV-conv-leq}
    \end{equation}
    Moreover, for every $\eps>0$, there exists $C(\eps)$
    such that the pointwise bound
    \begin{equation}
        |f_r\ast u(x)| \leq C r^{1-n} |Du|(B_{C(\eps)r}(x)) + \eps\|u\|_{L^\infty}.
        \label{local-conv-leq}
    \end{equation}
    holds for any $x\in\Real^n$ and $r>0$ and the local bounds
    \begin{align}
        \frac{1}{r}\|f_r\ast u\|_{L^1(A)}
        &\leq C |Du|(A+B_{C(\eps)r}) + \eps\|u\|_{BV}
        \label{local-l1-bd}\\
        \|f_r\ast u\|_{L^\infty(A)}
        &\leq C\|u\|_{L^\infty(A+B_{C(\eps)r})} + \eps\|u\|_{L^\infty}
        \label{local-linf-bd}
    \end{align}
    are satisfied for any measurable set $A\subset\Real^n$.
    \label{basic-ests}
\end{lemma}

\begin{proof}
The first estimate, Equation~\eqref{Linf-conv-leq}, follows from the fact
that $f\in L^1$ and the scale invariance $\|f_r\|_{L^1} = \|f\|_{L^1}$.

We first prove the bound~\eqref{BV-conv-leq} for smooth functions.  Using
the fundamental theorem of calculus,
\begin{equation*}
    |u(x-y)-u(x)| \leq \int_0^1 |\nabla u(x-ty)| |y|\,dt.
\end{equation*}
Combining this with the cancellation condition $\int f=0$, and then
applying the change of variables $w = ty$,
\begin{align}
    \begin{split}
        \frac{1}{r}|f_r\ast u(x)|
        &\leq \frac{1}{r}\int_{\Real^n} |u(x-y)-u(x)||f_r(y)|\,dy \\
        &\leq \frac{1}{r}\int_{\Real^n}
        \int_0^1|y||f_r(y)| |\nabla u(x-ty)|\,dt\,dy\\
        &= \frac{1}{r}
        \int_{\Real^n}\int_0^1 |f_r(w/t)|t^{-n-1}|w| |\nabla u(x-w)|\,dt\,dw.
    \label{fund-calc-leq}
    \end{split}
\end{align}
This is a convolution of $|\nabla u|$ against the function $g(r,w)$ defined
by
\begin{equation*}
    g(r,w) = r^{-1}|w| \int_0^1 t^{-n-1} |f_r(w/t)| \,dt.
\end{equation*}
The condition $|x|f\in L^1$ implies that $g\in L^1$.  Moreover we have the
scaling relationship that $g(r,w) = r^{-n} g(1,w/r)$.  Now Young's inequality
yields~\eqref{BV-conv-leq}.  To complete the proof for a general function
$u\in BV(\Real^n)\cap L^\infty(\Real^n)$, form a sequence of smooth
approximations $u_k$ such that $u_k\to u$ in $L^1$ and
$\|u_k\|_{BV} \to \|u\|_{BV}$.

We now prove the pointwise bound~\eqref{local-conv-leq}.
Let $\eps>0$, and choose $C(\eps)$ so large that
\begin{equation*}
    \int_{|x|>Cr} |f_r(x)|\,dx = \int_{|x|>C} |f(x)|\,dx < \eps / 2.
\end{equation*}

Using the cancellation condition again, and splitting the integral into parts,
\begin{align*}
    |f_r\ast u(x)| &\leq \int |u(x-y)-u(x)||f_r(y)|\,dy\\
    &\leq \int_{|y|<Cr} |u(x-y)-u(x)||f_r(y)|\,dy + \eps \|u\|_{L^\infty}.
\end{align*}
Define $\tilde{f}_r(y) = f_r(y) \One_{\{|y|<Cr\}}$, and notice that the
first term can be written as
\begin{equation*}
    \int_{\Real^n} |u(x-y)-u(x)||\tilde{f}_r(y)|\,dy.
\end{equation*}
Now applying again the fundamental theorem of calculus and a change of
variables as in~\eqref{fund-calc-leq}, we obtain
\begin{equation}
    \frac{1}{r} \int_{\Real^n} |u(x-y)-u(x)||\tilde{f}_r(y)|\,dy
    \leq \int_{\Real^n} \tilde{g}(r,w) |\nabla u|(x-w)\,dw
    \label{tildef-leq}
\end{equation}
where
\begin{equation*}
    \tilde{g}(r,w) = r^{-1}|w| \int_0^1 t^{-n-1} |\tilde{f}_r(w/t)| \,dt.
\end{equation*}
By scaling, $\tilde{g}(r,w) = r^{-n} g(1,w/r)$,
so $\|\tilde{g}(r,w)\|_{L^\infty} \leq C r^{-n}$.  Moreover
$\tilde{g}(r,w)$ has support in $|w|<Cr$.  Multiplying both sides
of~\eqref{tildef-leq} by $r$, we therefore obtain
\begin{align*}
    \int_{\Real^n} \tilde{g}(r,w) |\nabla u|(x-w) \,dw
    &= \int_{|x-w|<Cr} \tilde{g}(r,w) |\nabla u|(x-w)\,dw
    \\&\leq C r^{1-n} |Du|(B_{Cr}(x)).
\end{align*}

For the local $L^1$ bound~\eqref{local-l1-bd} we write
\[
    \frac{1}{r}|f_r\ast u(x)|
    \leq \int_{\Real^n} \tilde{g}(r,w)|\nabla u|(x-w)|\,dw
    + \int_{\Real^n} (g-\tilde{g})(r,w)|\nabla u|(x-w)|\,dw.
\]
Choosing $C(\eps)$ larger if necessary, we may assume that
$\|g-\tilde{g}\|_{L^1} < \eps$.  Now applying Young's inequality and
to both pieces and keeping track of the support yields~\eqref{local-l1-bd}.
Finally the local $L^\infty$ bound follows simply from the same type of
argument.  Writing
\[
    |f_r\ast u(x)| \leq |\tilde{f}_r\ast u| + |(f-\tilde{f})\ast u|.
\]
and applying Young's inequality to both sides completes the proof.
\end{proof}

We conclude the section with a few basic properties of the limit
of $\frac{1}{r}\|f_r\ast u\|_{L^2}^2$ that will let us perform various
decompositions.  The first property is a kind of locality.
\begin{proposition}
    For any closed set $K\subset\Real^n$,
    \[
        \limsup_{r\to 0} \frac{1}{r}\|f_r\ast u\|_{L^2(K)}^2
        \leq C \|u\|_{BV(K)} \|u\|_{L^\infty(K)}
    \]
    \label{local-limit-lem}
\end{proposition}
\begin{proof}
    Applying~\eqref{local-l1-bd} and~\eqref{local-linf-bd} we certainly have
    for each $\eps > 0$ that
    \[
        \limsup_{r\to 0} \frac{1}{r}\|f_r\ast u\|_{L^2(K)}^2
        \leq \|u\|_{BV(K+B_\eps)}\|u\|_{L^\infty(K+B_\eps)}
        + \eps \|u\|_{BV}\|u\|_{L^\infty}.
    \]
    By monotone convergence and the fact that $K = \bigcap_{\eps>0}(K+B_\eps)$,
    the result follows from taking $\eps\to 0$.
\end{proof}

The second result is that the functional is a continuity to perturbations.
\begin{proposition}
    Let $u\in BV\cap L^\infty(\Real^n)$.  For every $\eps>0$ there exists
    a $\delta>0$ such that if $v\in BV\cap L^\infty(\Real^n)$ satisfies
    \[
        \limsup_{r\to 0} \frac{1}{r}\|f_r\ast v\|_{L^2}^2 < \delta,
    \]
    then
    \[
        \limsup_{r\to 0}\frac{1}{r} \left|\|f_r\ast (u+v)\|_{L^2}^2
        - \|f_r\ast u\|_{L^2}^2\right| < \eps.
    \]
    \label{continuity-lem}
\end{proposition}
\begin{proof}
    Using linearity of convolution and expanding the square, one has
    \begin{align*}
        \frac{1}{r}\left| \|f_r\ast (u+v)\|_{L^2}^2
        - \|f_r \ast u\|_{L^2}^2\right|
        &\leq \frac{1}{r}\|f_r\ast v\|_{L^2}^2 +
        \frac{1}{r} 2\left|\langle f_r\ast u, f_r\ast v\rangle\right|.
    \end{align*}
    The first term is bounded by $2\delta$ for sufficiently small $r$.  We may
    apply Cauchy-Schwartz to the second term to obtain the bound
    \[
        \frac{1}{r}|\langle f_r\ast u, f_r\ast v\rangle|
        \leq \left(\frac{1}{r^{1/2}}\|f_r\ast u\|_{L^2}\right)
        \left(\frac{1}{r^{1/2}}\|f_r\ast v\|_{L^2}\right).
    \]
    By the global $L^1$ and $BV$ estimates of Lemma~\ref{basic-ests} the first
    term remains bounded, while the second term is eventually bounded by
    $2\sqrt{\delta}$ by hypothesis.  Thus one can choose $\delta$ sufficiently
    small to conclude.
\end{proof}

\section{An example with infinite perimeter.}
\label{construction-section}

In this section we show that
$\liminf_{\eps\to 0}\frac{1}{|\log\eps|}\|\gamma_\eps\ast \One_E\|_{H^{1/2}}^2$
does not characterize sets of finite perimeter.  Indeed we show that one
can find a set with infinite perimeter for which the $\liminf$ is zero.
\begin{theorem}
    For any $n>0$ there exists $E\subset \Real^n$ with $P(E)=\infty$ and such
    that
    \[
        \liminf_{\eps\to 0}
        \frac{1}{|\log\eps|}\|\gamma_\eps\ast \One_E\|_{H^{1/2}}^2 = 0.
    \]
    \label{example-thm}
\end{theorem}

This will follow from the result in one dimension.
\begin{lemma}
    There exists a set $E\subset[0,1]$ with $0<|E|<1$ and
    \[
        \liminf_{\eps\to 0}
        \frac{1}{|\log\eps|}\|\gamma_\eps\ast \One_E\|_{H^{1/2}}^2 = 0.
    \]
    \label{1d-example-lem}
\end{lemma}

We show now that Theorem~\ref{example-thm} follows from the one-dimensional
case.
\begin{proof}[Proof of Theorem~\ref{example-thm} using Lemma~\ref{1d-example-lem}]
    Choose $E\subset[0,1]$ according to Lemma~\ref{1d-example-lem}.  For
    $n>1$, consider the Cartesian product $E^n\subset\Real^n$.  That is,
    the indicator function can be written
    \[
        \One_{E^n}(x_1,\ldots,x_n) = \One_E(x_1)\One_E(x_2)\cdots\One_E(x_n).
    \]
    We can use the fact that the Gaussian separates to estimate
    $\gamma_\eps\ast\One_{E^n}$:
    \begin{align*}
        \|\gamma_{\eps}\ast \One_{E^n}\|_{H^{1/2}}^2
        &= \int\cdots\int \left(\sum_j\xi_j^2\right)^{1/2} \,\,\prod_{i=1}^n
        \left|\widehat{\gamma_\eps}(\xi_i)\widehat{\One_E}(\xi_i)\right|^2
        \,d\xi_i.
        \\&\leq
        \sum_{j=1}^n\int\cdots\int |\xi_j| \,\,\prod_{i=1}^n
        \left|\widehat{\gamma_\eps}(\xi_i)\widehat{\One_E}(\xi_i)\right|^2
        \,d\xi_i.
        \\&= n \|\gamma_\eps\ast \One_E\|_{H^{1/2}}^2
        \|\gamma_\eps\ast \One_E\|_{L^2}^{2(n-1)}\\
        &\leq n\|\gamma_\eps\ast\One_E\|_{H^{1/2}}^2.
    \end{align*}
    In the last step we used the fact that $|E|<1$.
\end{proof}
The rest of the section will be devoted to proving Lemma~\ref{1d-example-lem}.

\subsection{Plan for the construction}
The idea of the construction is to design a sequence of smooth functions
$\phi_k$ that act as the smoothed versions of the set $E$ at varying scales.
That is, we would like
\[
\gamma_\delta \ast \One_E \approx \gamma_\delta \ast \phi_k
\]
for any $\delta \geq \delta_k$, where $\delta_k$ is a sequence of scales
converging to zero.  If the scales $\delta_k$ are sufficiently small, then
because $\phi_k$ are smooth we should have
\[
    \frac{1}{|\log\eps|} \|\gamma_{\delta_k}\ast \One_E\|_{H^{1/2}}^2
    \approx
    \frac{1}{|\log\eps|} \|\gamma_{\delta_k}\ast \phi_k\|_{H^{1/2}}^2
    \approx 0.
\]
To ensure that the smooth functions $\phi_k$ converge to a measurable set, we
enforce that $0\leq \phi_k\leq 1$ and that the sets $\{\phi_k = 0\}$ and
$\{\phi_k = 1\}$ are strictly increasing.  Moreover the functions
$\phi_k$ face a compatibility condition whereby local averages of $\phi_{k+1}$
must match local averages of $\phi_k$.
The compatibility condition is of the form
\[
    \gamma_{\delta_k} \phi_{k+1} \approx \gamma_{\delta_k} \phi_k.
\]
To reconcile this with our need for the set $\{\phi_{k+1}\in\{0,1\}\}$ to
increase, we construct $\phi_{k+1}$ to be highly oscillatory compared to the
scale $\delta_k$, so that the $\delta_k$ smoothing recovers only the
smoother function $\phi_k$.  A cartoon of the first step of the construction is
given in Figure~\ref{example-cartoon-fig}.

\begin{figure}
    \begin{center}
        \includegraphics{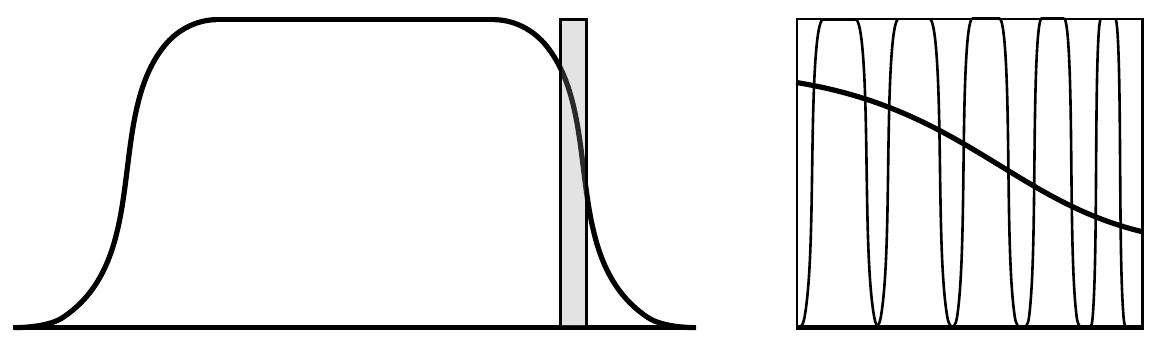}
    \end{center}
    \caption{The beginning of the compatible sequence $\phi_k$.  On the left,
    a smooth function $\phi_1$ is chosen which takes values in $[0,1]$.  On the
right, the function $\phi_2$ is depicted on a magnified portion of the interval
(colored in gray on the left).  In this interval, $\phi_2$ oscillates between $0$
and $1$ in such a way to preserve the local averages of $\phi_1$.}
    \label{example-cartoon-fig}
\end{figure}

The following definition quantifies the conditions outlined above.
\begin{definition}
    Let $\phi_k\in C_c^\infty((0,1))$ be a sequence of smooth functions
    and let $\eps_k>0$ be a decreasing sequence of scales with $\eps_k\to 0$.
    We say that $\phi_k$ is a \emph{compatible sequence} if the following
    properties hold:
    \begin{itemize}
        \item \emph{Nontriviality:} \[0< |\{\phi_1 = 1\}|.\]
        \item \emph{Convergence to a set:}
            \begin{equation}
                \begin{split}
                    |\{\phi_k \not\in \{0,1\}\}| &< (0.99)^k, \\
                    \{\phi_k = 1\} &\subset \{\phi_{k+1} = 1\}, \\
                    \{\phi_k = 0\} &\subset \{\phi_{k+1} = 0\},
                    \qquad\text{and} \\
                    0\leq &\phi_k \leq 1.
                \end{split}
            \end{equation}
        \item \emph{Smoothness to scale $\eps_k$:}
            \begin{equation}
                \frac{1}{|\log\eps_k|}
                \|\gamma_{\eps_k}\ast \phi_k\|_{H^{1/2}}^2 < 2^{-n}.
                \label{smooth-scale-r-bd}
            \end{equation}
        \item \emph{Compatibility across scales:}  One has
            \begin{equation}
                \|\gamma_r\ast (\phi_k-\phi_{k+1})\|_{L^2} < \eps_k^{3}2^{-k}
                \label{compatibility-bd}
            \end{equation}
            for all $r \geq \eps_k$.
    \end{itemize}
\end{definition}

\begin{lemma}
    \label{set-from-sequence}
    Let $\phi_k$ be a compatible sequence with scales $\eps_k$.  Then
    there exists a measurable set $E\subset [0,1]$ with $0<|E|<1$,
    $\phi_k\to E$ in $L^p$ for all $p<\infty$, and
    \[
        \lim_{k\to\infty}
        \frac{1}{|\log\eps_k|}\|\gamma_{\eps_k}\ast \One_E\|_{H^{1/2}}^2 = 0.
    \]
\end{lemma}
\begin{proof}
    The existence of the limiting set $E$ follows straightforwardly
    from the first two hypotheses on $\phi_k$.

    To show that $\|\gamma_{\eps_k}\ast \One_E\|_{H^{1/2}}^2$ grow manageably,
    we write $\One_E$ as a telescoping sum and apply the triangle inequality:
    \begin{align*}
        \frac{1}{|\log\eps_k|}\|\gamma_{\eps_k}\ast\One_E\|_{H^{1/2}}^2
        &\leq
        \frac{2}{|\log\eps_k|}\|\gamma_{\eps_k}\ast\phi_k\|_{H^{1/2}}^2
        + \frac{2}{|\log\eps_k|}\left(\sum_{m=k}^\infty
        \|\gamma_{\eps_k}\ast (\phi_{m+1}-\phi_m)\|_{H^{1/2}}\right)^2
    \end{align*}
    The first term we bound using~\eqref{smooth-scale-r-bd}.  For the terms in
    the sum we interpolate between $L^2$ and $H^1$,
    \begin{align*}
        \|\gamma_{\eps_k}\ast (\phi_{m+1}-\phi_m)\|_{H^{1/2}}
        &\leq
        \|\gamma_{\eps_k}\ast (\phi_{m+1}-\phi_m)\|_{L^2}^{1/2}
        \|\gamma_{\eps_k}\ast (\phi_{m+1}-\phi_m)\|_{H^1}^{1/2}
        \\ &\leq
        (\eps_m^3 2^{-m})^{1/2}
        (2\|\gamma_{\eps_k}\|_{H^1})^{1/2}
        \\ & \leq
        C(\eps_m^3 2^{-m})^{1/2} (\eps_k)^{-3/2}
    \end{align*}
    The bound on the $H^1$ norm follows from the fact that
    $\|\phi_{m+1}-\phi_m\|_{L^1} < 2$.
    Since $\eps_m < \eps_k$ for $m>k$, this is bounded by $2^{-m/2}$, and is
    thus clearly summable over all $m\geq k$.
\end{proof}

Thus our task is reduced to showing the existence of a compatible set.
This will be done inductively in the next subsection.

\subsection{Technical constructions}
In this section we prove three short facts that allow us to construct
$\phi_{n+1}$ from $\phi_n$.  The first says that, in order to get the local
approximation $\gamma_r\ast \phi_k \approx \gamma_r\ast \phi_{k+1}$, it
suffices to demonstrate that the averages of $\phi_k$ and $\phi_{k+1}$ are
equal on many short intervals.  Then we show how to actually construct a
function $\phi_{n+1}$ from $\phi_n$ such that the averages on short intervals
are correct, with the constraint that $\phi_{n+1}$ takes values in $\{0,1\}$
more often.  This is done with the help of a short proposition that takes care of the
case of one single interval.
\begin{proposition}
    Let $\phi \in C_c^\infty((0,1))$ with $0\leq\phi\leq 1$,
    $\delta>0$ be a scale, $\eps>0$ be some tolerance, and $k>0$ be an
    integer.
    Then for sufficiently large $N$ we have the following:
    For every $\psi\in C_c^\infty((0,1))$ with $0\leq \psi\leq 1$, if
    \begin{equation}
        \phi(\frac{i}{N})
        = N\cdot\int_{i/N}^{(i+1)/N}\psi(t)\,dt
        \label{interval-avg-eq}
    \end{equation}
    for all but at most $k$ values of $i\in [N]$, then
    \[
        \|\gamma_r \ast (\phi-\psi)\|_{L^\infty} < \eps.
    \]
    for all $r > \delta$.
    \label{local-avg-constraints}
\end{proposition}
\begin{proof}
    Let $\chi$ be the indicator function for the interval $[0,1]$.
    We will consider the function $g=\chi_{M/N}\ast \psi$.  We show that
    by choosing $M$ large enough, and then $N$ to be a sufficiently large
    multiple of $M$, we can enforce $\|g-\phi\|_\infty < \eps/2$.

    Indeed, since~\eqref{interval-avg-eq}
    holds on all but at most $k$ intervals,
    \[
        \left|g\left(\frac{i}{N}\right) - \frac{1}{M}\sum_{j=i}^{i+M-1}
        \phi\left(\frac{j}{N}\right)\right| \leq \frac{k}{M}.
    \]
    Assuming we take $M$ to be sufficiently large, and then
    $N$ to be a sufficiently large multiple of $M$,
    we have $|g(i/N)-\phi(i/N)|<\eps/10$.
    The definition of $g$ also yields the Lipschitz bound
    $|g'(x)|\leq 2NM^{-1}$,
    so that we can conclude $\|g-\phi\|_\infty\leq \eps/2$ as desired,
    provided again $M$ is large enough.

    Finally, we take $N$ large enough that
    \[
    \|\gamma_\delta - \chi_{M/N}\ast \gamma_\delta\|_{L^\infty} < \eps/100
    \]
    and
    \[
    \|\phi - \chi_{M/N}\ast \phi\|_{L^\infty} < \eps/100.
    \]
    Then we conclude since
    \begin{align*}
        \|\gamma_\delta\ast (\phi-\psi)\|_{L^\infty} &\leq
        \|(\gamma_\delta-\gamma_\delta\ast \chi_{M/N})\ast (\phi-\psi)\|_{L^\infty} +
        \|\gamma_\delta \ast (\chi_{M/N}\ast (\phi-\psi))\|_{L^\infty}
        \\&< \eps.
    \end{align*}
\end{proof}

The next proposition lets us satisfy the local average condition on an interval
with a function that looks more like an indicator function.
\begin{proposition}
    Let $a\in [0,1)$ be a target average.  Then there exists a
    function $\psi\in C_c^\infty((0,1))$ satisfying
    $|\{\psi\in\{0,1\}| > 0.1$ and $\int_0^1\psi = a$, and
        such that the sets $\{\psi=0\}$ and $\{\psi=1\}$ are unions
        of finitely many closed intervals.
        \label{interval-avg-construction}
\end{proposition}
\begin{proof}
    We will split into the cases $a>1/2$ and $a<1/2$.  We begin with
    the case $a>1/2$.
    Let $\sigma$ be a smooth increasing function satisfying
    $\sigma(x) = 0$ for $x\leq 1/2$ and $\sigma(x)=1$ for $x\geq1$.
    k

    Consider the following function $\psi_t\in C_c^\infty((0,1))$
    defined for $0<t\leq1/2$:
    \[
        \psi_t(x) = \begin{cases}
            \sigma(x/t), & x < t \\
            1, & x\in [t,1-t] \\
            \sigma((1-x)/t), & x > 1-t\\
        \end{cases}.
    \]
    Each of $\psi_t$ satisfy the condition $|\{\psi\in\{0,1\}\}|>0.1$.
    Moreover $I(t) = \int_0^1 \psi_t$ is a continuous function in $t$
    with $I(1/2) < 1/2$ and $\lim_{t\to 0} I(t) = 1$.  Thus by the
    intermediate value theorem we have that, for any $a\in (1/2,1)$,
    there exists $t$ such that $\int \psi_t = a$.

    Now suppose $a\leq 1/2$.  Observe that $I(0.1) > 1/2$, so
    the function $\psi = \frac{a}{I(0.1)}\psi_{1/10}$ satisfies the
    constraints.
\end{proof}

Finally we use Proposition~\ref{interval-avg-construction} to
construct a function satisfying the local average constraints of
Proposition~\ref{local-avg-constraints}.
\begin{proposition}
    Let $\phi\in C_c^\infty((0,1))$ satisfy $0\leq \phi\leq 1$, and let
    $N > 0$.
    Suppose that the sets $\{\phi=1\}$ and $\{\phi=0\}$ can be written as
    unions of at most $k$ intervals.  Then there exists
    $\psi\in C_c^\infty((0,1))$ satisfying the following constraints:
    \begin{itemize}
        \item $\{\phi=1\} \subset \{\psi=1\}$ and
            $\{\phi=0\} \subset \{\psi=0\}$.
        \item $|\{\psi\not\in\{0,1\}\}|\leq 0.99 \cdot |\{\phi\not\in\{0,1\}\}|$.
        \item The level sets $\{\psi=0\}$ and $\{\psi=1\}$ can each be
            written as a union of finitely many intervals.
        \item The function $\psi$ satisfies the following local average
            constraints
            \[
                \phi\left(\frac{i}{N'}\right)
                = N'\int_{i/N'}^{(i+1)/N'} \psi(t)\,dt
            \]
            for some $N'>N$, and on all but at most $2k$ intervals.
    \end{itemize}
    \label{compatible-construction}
\end{proposition}
\begin{proof}
    Choose $N'>N$ such that each interval $[i/N',(i+1)/N']$ contains
    at most point in $\partial \{\phi=0\} \cup \partial\{\phi=1\}$.
    Let $I_i$ be the interval $[i/N',(i+1)/N']$.
    Let $A$ be the set of indices such that their intervals contain
    such an endpoint, that is
    \[
        A := \left\{i ; I_i \cap
        \left(\partial \{\phi=0\} \cup \partial\{\phi=1\}\right) \not=
    \emptyset\right\}.
    \]
    We will define functions $F_i\in C^\infty([0,1])$ for
    $0\leq i\leq N$ and set
    \[
        \psi(x) = F_{\lfloor N\cdot x\rfloor}(\mathrm{frac}(N\cdot x))
    \]
    where $\mathrm{frac}(x)$ denotes the fractional part of $x$.
    We split the choice of $F_i$ into three cases.

    \vskip 10pt
    \noindent\emph{Case I:} $i\not\in A$ and $I_i \subset \{\phi=1\}$.
    In this case we simply set $F_i = 1$.

    \vskip 10pt
    \noindent \emph{Case II:} $i\not\in A$ and $\phi(i/N') < 1$.
    Simply use Proposition~\ref{interval-avg-construction} to choose
    $F_i$ such that $\int F_i = \phi(i/N')$.

    \vskip 10pt
    \noindent \emph{Case III:} $i\in A$.
        Choose any $F_i$ subject to the constraints
            $0\leq F_i\leq 1$, $\phi\in C_c^\infty$,
            $\{\phi=1\}\subset \{\psi=1\}$ and
            $\{\phi=0\}\subset \{\psi=0\}$.

    \vskip 12pt
    Our choice of $N'$ guarantees that the above three cases are exhaustive.
    The resulting function $\psi$ satisfies all the conditions of the lemma.
\end{proof}

\subsection{The iterative algorithm}
In this section we combine the main lemmas above to inductively define
a compatible sequence $\phi_n$.
\begin{proof}[Proof of Lemma~\ref{1d-example-lem}]
    Using Lemma~\ref{set-from-sequence}, it suffices to construct a
    compatible sequence.
    We begin with any valid function
    $\phi_1\in C_c^\infty((0,1))$ satisfying the nontriviality
    constraint $|\{\phi=1\}| > 0$ and such that the sets
    $\{\phi_1=1\}$ and $\{\phi_1=0\}$ are finite unions of closed intervals.
    Since $\phi_1$ is smooth, and thus in $H^{1/2}$
    we have
    \[
        \lim_{\eps\to 0^+}
        \frac{1}{|\log\eps|}\|\gamma_\eps\ast \phi_1\|_{H^{1/2}}^2 \to 0,
    \]
    so we may choose $\eps_1$ small enough to satisfy the smoothness
    constraint~\eqref{smooth-scale-r-bd}.

    We now induct on $k$.
    Suppose that the sets $\{\phi_k=1\}$ and $\{\phi_k=0\}$ are
    unions of at most $K$ intervals. Applying
    Proposition~\ref{local-avg-constraints} with $\phi=\phi_k$,
    $\delta = \eps_k$, $\eps=\eps_k^3 2^{-k}$, and $k=K$, we obtain
    a value $N_k$ for which the interval average
    constraints~\eqref{interval-avg-eq} imply the compatibility
    bound~\eqref{compatibility-bd}.  We can then use
    Proposition~\ref{compatible-construction} with $\phi=\phi_k$ and $N_k$
    to construct $\phi_{k+1}$.  The function $\phi_{k+1}$ is smooth,
    so we can find $\eps_{k+1}$ to satisfy~\eqref{smooth-scale-r-bd}.  Moreover,
    we have that the sets $\{\phi_{k+1}=1\}$ and $\{\phi_{k+1}=0\}$ are finite
    unions of closed intervals, so the induction is closed.
\end{proof}

\section{Characterizing Sets of Finite Perimeter}
\label{characterize-sec}

\subsection{The lower bound}
\label{characterization}
In this section we prove the following characterization of sets of finite
perimeter.
\begin{theorem}
    Let $E\subset\Real^n$ be a set with $P(E)=\infty$.  Then
    \[
        \limsup_{\eps\to 0}\frac{1}{|\log\eps|}
        \|\gamma_\eps\ast \One_E\|_{H^{1/2}}^2 = \infty.
    \]
    \label{limsup-thm}
\end{theorem}
The proof of this theorem goes through an analysis of the smoothed functions
$\gamma_\eps\ast\One_E$.  The difficulty is that these functions may be so
smooth that $\|\gamma_\eps\ast\One_E\|_{H^{1/2}}^2$ could be very small.
However, using a characterization of sets of finite perimeter due to
Ambrosio, Bourgain, Brezis, and Figalli, we will be able to see that
\[
    \eps^{-1}\|\One_E - \gamma_\eps\ast\One_E\|_{L^2}^2
\]
grows to be large if $E$ is a set of infinite perimeter~\cite{ambrosio2016bmo}.
Decomposing the difference $\One_E - \gamma_\eps\ast\One_E$ over many
scales in the Fourier domain, we will see that there must be at least some
wavelength $\eps'<\eps$ that contributes significantly to the difference.
It is at this wavelength that $\|\gamma_{\eps'}\ast\One_E\|_{H^{1/2}}^2$ is
large.

To make this analysis convenient, we will
make our smoothing kernels compactly supported in Fourier space.  That is, let
$\psi\in C_c^\infty(\Real)$ have support in $[-1,1]$ with
$\psi(\xi) = 1$ for $|\xi|<1/2$.  Then by construction, the differences
$(\psi_r-\psi_{r/2})\ast \One_E$ and $(\psi_h-\psi_{h/2})\ast \One_E$
are orthogonal so long as $r\not\in (h/4,4f)$.  From this we deduce the following
approximate orthogonality property:
\begin{equation}
    \|\psi_r\ast \One_E - \One_E\|_{L^2}^2
    \leq C \sum_{k=0}^\infty \|\psi_{r/2^{k}}\ast \One_E -
    \psi_{r/2^{k+1}}\ast \One_E\|_{L^2}^2.
    \label{perp-bd}
\end{equation}
Next we demonstrate the connection between these differences and the quantity
$\|\gamma_\eps\ast \One_E\|_{H^{1/2}}$ via the kernel described
in~\eqref{special-kernel-defn}.
\begin{proposition}
    With $\varphi$ as defined by Equation~\eqref{special-kernel-defn}, we
    have
    \[
        \|(\psi_r-\psi_{r/2})\ast \One_E\|_{L^2}^2 \leq C
        \|\varphi_{r}\ast \One_E\|_{L^2}^2
    \]
    for any measurable set $E\subset\Real^n$.
    \label{psi-phi-bd}
\end{proposition}
\begin{proof}
    By Plancherel's theorem and homogeneity it suffices to check that
    there exists some constant $C$ such that
    \[
        \left|\widehat{\psi}(\xi) - \widehat{\psi}(\xi)\right|^2
        \leq C\left|\widehat{\varphi}(\xi)\right|^2
    \]
    for all $\xi\in\Real^n$.  By construction of $\varphi$, the left hand side
    has support in the annulus $\frac{1}{4}\leq r|\xi|\leq 1$.
    The result follows from the compactness of the annulus and the positivity
    of $\widehat{\varphi}$ on $\Real^n\setminus\{0\}$.
\end{proof}

\begin{lemma}
    Suppose that $E\subset\Real^n$ satisfies
    \[
        \limsup_{\eps\to 0}\frac{1}{|\log\eps|}\|\gamma_\eps\ast \One_E\|_{H^{1/2}}^2
        < \infty.
    \]
    Then
    \begin{equation}
        \liminf_{n\to\infty} 2^n \sum_{k=n}^\infty
        \|\varphi_{2^{-k}}\ast \One_E\|_{L^2}^2 < \infty.
        \label{inf-energy-bd}
    \end{equation}
    \label{inf-energy-lemma}
\end{lemma}
\begin{proof}
    Using the definition of $\varphi$, the condition on $E$ implies that there
    exists $C$ such that for every integer $n>0$,
    \begin{equation}
        \sum_{k=1}^n 2^k\|\varphi_{2^{-k}} \ast \One_E\|_{L^2}^2 < Cn.
        \label{dyadic-H12-bd}
    \end{equation}
    We first use this inequality to bound the infinite sum
    in~\eqref{inf-energy-bd} in terms of a finite one.  Indeed, by grouping
    the infinite sum into dyadic pieces and applying the bound above in each
    piece, we have
    \begin{align*}
        2^n\sum_{k=2n}^\infty \|\varphi_{2^{-k}}\ast \One_E\|_{L^2}^2
        &\leq 2^n
        \sum_{N=1}^\infty \sum_{k=n2^N}^{n2^{N+1}}\|\varphi_{2^{-k}}\ast\One_E\|_{L^2}^2
        \\&
        \leq 2^n\sum_{N=1}^\infty 2^{-n2^N}
        \sum_{k=n2^N}^{n2^{N+1}} 2^k\|\varphi_{2^{-k}}\ast\One_E\|_{L^2}^2
        \\&
        \leq C 2^n\sum_{N=1}^\infty 2^{-n2^N} n2^{N+1}
    \end{align*}
    which is clearly bounded independently of $n$.  Thus, it suffices to
    show that
    \[
        \liminf_{n\to \infty} 2^n \sum_{k=n}^{2n}
        \|\varphi_{2^{-k}}\ast \One_E\|_{L^2}^2 < \infty.
    \]
    We do this by finding, for each $n>0$, a suitable scale $n\leq m\leq 2n$.
    To see that at least one $m$ suffices we average over all such scales:
    \begin{align*}
        \frac{1}{n} \sum_{m=n}^{2n}
        2^m\sum_{k=m}^{2m}\|\varphi_{2^{-k}}\ast \One_E\|_{L^2}^2
        &\leq
        \frac{1}{n} \sum_{k=n}^{4n}
        \|\varphi_{2^{-k}}\ast \One_E\|_{L^2}^2
        \sum_{m=0}^k 2^m \\
        &\leq
        \frac{2}{n} \sum_{k=n}^{4n} 2^k\|\varphi_{2^{-k}}\ast \One_E\|_{L^2}^2.
    \end{align*}
    The last sum is bounded by using again~\eqref{dyadic-H12-bd}.  Thus it is
    possible to find $m>n$ such that
    \[
        2^m\sum_{k=m}^{2m} \|\varphi_{2^{-k}}\ast \One_E\|_{L^2}^2
    \]
    is bounded independent of $n$.
\end{proof}

The following lemma is a characterization of sets of finite perimeter
that appears in~\cite{ambrosio2016bmo} that we will rely on.  A $\delta$-cube
is any cube in $\Real^n$ with side length $\delta$.
\begin{lemma}[{\cite[Lemma 3.2]{ambrosio2016bmo}}]
    \label{intermediate-set-lemma}
    Let $K>0$ and $E\subset\Real^n$ be a measurable set with $P(E)=\infty$.  Then
    there exists $\delta_0=\delta_0(K,A)$ such that for every $\delta<\delta_0$
    it is possible to find a disjoint collection $\mathcal{U}_\delta$ of
    $\delta$-cubes $Q'$ with $\#\mathcal{U}_\delta > K\delta^{-n+1}$ and
    \[
        2^{-n-1} \leq \frac{|Q'\cap E|}{|E|} \leq 1-2^{-n-1}
    \]
    for every $Q'\in\mathcal{U}_\delta$.
\end{lemma}

\begin{proposition}
    Let $Q=(-\frac{1}{2},\frac{1}{2})^n\subset \Real^n$ be the unit cube.
    Suppose that $E\subset\Real^n$ is a measurable set with
    $2^{-n-1}\leq |E\cap Q|\leq 1-2^{-n-1}$.  Then there exists constants
    $c_n,r_n>0$ such that
    \[
        \|\psi_{r_n}\ast \One_E - \One_E\|_{L^2(Q)}^2 > c_n.
    \]
    \label{find-big-ball}
\end{proposition}
\begin{proof}
    We choose $r_n$ so small that
    \[
        \|\psi_{r_n}\ast \One_Q - \One_Q\|_{L^1} < 2^{-n-2}.
    \]
    With this choice for $r_n$,
    \begin{align*}
        \left|\int_Q \psi_{r_n}\ast\One_E(x)\,dx
        - |E\cap Q|\right| &=
        \left|\int \One_E(x) \left(\psi_{r_n}\ast \One_Q(x) - \One_Q(x)\right)\,dx
        \right|
        \\&<2^{-n-2}.
    \end{align*}
    It follows from the continuity of $\psi_{r_n}\ast \One_E$ that for some
    point $x_0\in Q$, $\psi_{r_n}\ast \One_E(x_0) \in (2^{-n-2},1-2^{-n-2})$.
    Since $\|\nabla\psi_{r_n}\|_{L^\infty} < r_n^{-1}$,
    one has
    \[
        2^{-n-2}\leq \psi_{r_n}\ast \One_E(y) < 1-2^{-n-2}
    \]
    for any $|y-x_0| < 2^{-n-2}r_n$.  In particular,
    \[
        |\psi_{r_n}\ast \One_E(y) - \One_E(y)| > 2^{-n-2}
    \]
    for $y\in B_{2^{-n-2}r_n}$.  The claim follows upon integrating the above
    bound over $B_{2^{-n-2}r_n}\cap Q$.
\end{proof}
The above lemmas combine in a straightforward manner to prove our main result
for this section.
\begin{proof}[Proof of Theorem~\ref{limsup-thm}]
    Suppose that $E\subset\Real^n$ is a set with
    \[
        \limsup_{\eps\to 0}
        \frac{1}{|\log\eps|}\|\gamma_\eps\ast \One_E\|_{H^{1/2}}^2 < \infty.
    \]
    We will then show that
    \begin{equation}
        \liminf_{\delta>0} \delta^{-1} \|\psi_\delta\ast \One_E-\One_E\|_{L^2}^2
        < \infty.
        \label{liminf-bd}
    \end{equation}
    To show that this implies that $E$ is a set of finite perimeter,
    let $\delta > 0$ be very small and such that
    \[
        \delta^{-1} \|\psi_\delta\ast \One_E-\One_E\|_{L^2}^2 < C.
    \]
    Consider a collection $\mathcal{U}_{\delta/r_n}$ of
    $\delta/r_n$-cubes such that $|Q'\cap E|/|Q'| \in (2^{-n-1},1-2^{-n-1})$
    for all $Q'\in\mathcal{U}_\delta$.  Appropriately scaling the
    conclusion of Proposition~\ref{find-big-ball},
    \[
        \|\psi_{\delta}\ast \One_E - \One_E\|_{L^2(Q')}^2 > \delta^n c_n
    \]
    for all $Q'\in\mathcal{U}_\delta$.  In particular,
    \[
        \delta^n c_n \#\mathcal{U}_\delta <
        \|\psi_\delta \ast \One_E - \One_E\|_{L^2}^2 < \delta C.
    \]
    Thus $\#\mathcal{U}_\delta < K\delta^{1-n}$ for some $K$.  Since this holds
    for arbitrarily small $\delta$, it follows from
    Lemma~\ref{intermediate-set-lemma} that $P(E)<\infty$.

    Now we prove~\eqref{liminf-bd}.  Indeed, according to
    Lemma~\ref{inf-energy-lemma}, we may find $C>0$ and arbitrarily large $n$
    such that
    \[
        2^n\sum_{k=n}^\infty \|\varphi_{2^{-k}}\ast \One_E\|_{L^2}^2 < C.
    \]
    Setting $\delta= 2^{-n}$ and applying Proposition~\ref{psi-phi-bd} and the
    orthogonality property~\eqref{perp-bd} we obtain
    \begin{align*}
        \|\psi_{2^{-n}}\ast \One_E - \One_E\|_{L^2}^2
        &\leq C \sum_{k=n}^\infty \|\psi_{2^{-k}}\ast \One_E
        -\psi_{2^{-k-1}}\ast \One_E\|_{L^2}^2 \\
        &\leq C \sum_{k=n}^\infty \|\varphi_{2^{-k}}\ast \One_E\|_{L^2}^2 \\
        &\leq 2^{-n}C
    \end{align*}
    as desired.
\end{proof}

\section{Approximating the Perimeter of a Set}
\label{sofp-sec}

In this section we consider the case where $E\subset \Real^n$ is a set
of finite perimeter and calculate the limit
$\lim_{r\to 0^+} \frac{1}{r}\|f_r\ast \One_E\|_{L^2}^2$ for kernels
$f$ which satisfy the conditions of Lemma~\ref{basic-ests}.
\begin{theorem}
  \label{sofp-thm}
  Let $f\in C^\infty(\Real^n)$ be a smooth function with $\int f =0$ and
  $\int |x||f|\,dx < \infty$.  Then for any set $E\subset\Real^n$ of finite
  perimeter,
  \begin{equation}
    \lim_{r\to 0^+} \frac{1}{r}\|f_r\ast \One_E\|_{L^2}^2
    = \int_{\partial^*E} F(\nu) \,d\Hausdorff^{n-1}.
    \label{sofp-eq}
  \end{equation}
  where the function $F\in C^{0,1}(\Sphere^{n-1})$ is given
  by the expression in~\eqref{Fnu-defn-eq}.
\end{theorem}
\begin{remark}
    Choosing $f=\varphi$ as defined in~\eqref{special-kernel-defn}, we can
    conclude that for a set of finite perimeter $E$,
    \[
        \lim_{\eps\to 0}\frac{1}{|\log\eps|}
        \|\gamma_\eps\ast \One_E\|_{H^{1/2}}^2 = c_n P(E).
    \]
\end{remark}

The proof of Theorem~\ref{sofp-thm} consists of
a covering argument using De Giorgi's structure theorem,
which reduces the problem to the following local computation.
The local computation takes place over small balls around points in the
reduced boundary $\partial^*E$, where $E$ looks like a half-plane.  To set
up the notation, for each $\nu\in\Sphere^{n-1}$ define the half plane
\[
    H_\nu := \{x\in\Real^n \mid x\cdot\nu \leq 0\},
\]
and let $B$ be the unit ball centered at the origin.

\begin{lemma}
  Let $f$ be as in Theorem~\ref{sofp-thm}, and let $\eps>0$.
  Then there exists $\delta>0$ and $\alpha>0$ such that the following holds:
  If $E\subset\Real^n$ is a set of finite perimeter which is
  close to a half-space $H_{\nu_0}$ in the sense that
  $|(E\Delta H_{\nu_0})\cap B| < \delta$ and
  $P(E;B)-\omega_{n-1}<\delta$, then
  \[
    \limsup_{r\to 0^+}\left|
    \frac{1}{r}\|f_r\ast \One_E\|_{L^2(B)}^2 -
    \int_{\partial^*E}F(\nu)\,d\Hausdorff^{n-1}\right| < \eps^\alpha.
  \]
  where $\nu$ denotes the unit normal to $\partial^*E$.
  \label{local-lem}
\end{lemma}

Before proving this lemma, we first show how it implies Theorem~\ref{sofp-thm}
\begin{proof}[Proof of Theorem~\ref{sofp-thm} using Lemma~\ref{local-lem}]
  Let $\eps>0$, and choose $\delta>0$ according to Lemma~\ref{local-lem}.
  For each $x\in\partial^*E$ and $\rho>0$, we say that the ball $B_\rho(x)$
  is $\delta$-acceptable if
  \begin{align}
    \begin{split}
      |(E \Delta H_\nu(x))\cap B_\rho(x))|
       &< \rho^n \delta \\
      |P(E;B_\rho(x)) - \omega_{n-1}\rho^{n-1}| &< \rho^{n-1} \delta.
      \label{delta-acceptabla}
    \end{split}
  \end{align}
  De Giorgi's structure theorem ensures that the set
  $\mathcal{B}_\delta$ of $\delta$-acceptable balls is a Vitali covering
  of $\partial^* E$.  Thus we may choose a finite set
  $\{B_{\rho_j}(x_j)\}_{j=1}^N$ of disjoint balls such that
  \begin{equation*}
    \Hausdorff^{n-1}(\partial^* E \setminus
    \bigcup_j B_{\rho_j}(x_j)) < \eps.
  \end{equation*}
  Let $R$ denote the remainder set $\Real^n \setminus \bigcup_{j=1}^N B_{\rho_j}(x_j)$.
  The set $R$ is closed and $\|\One_E\|_{BV(R)}=P(E;R)<\eps$,
  so Proposition~\ref{local-limit-lem} guarantees that
  \begin{equation*}
    \limsup_{r\to 0} \frac{1}{r}\|f_r\ast \One_E\|_{L^2(R)}^2 < C\eps,
  \end{equation*}
  and so trivially
  \begin{equation*}
    \limsup_{r\to 0} \left|
    \frac{1}{r}\|f_r\ast \One_E\|_{L^2(R)}^2 -
    \int_{\partial^* E\cap R} F(\hat{n})\,d\Hausdorff^{n-1}\right| < C\eps^\alpha.
  \end{equation*}
  On each $B_{\rho_j}(x_j)$, scaling and applying Lemma~\ref{local-lem} yields
  \begin{equation*}
    \limsup_{r\to 0}
    \left|\frac{1}{r}\|f_r\ast \One_E\|_{L^2(B_{\rho_j}(x_j))}^2
    - \int_{\partial^*E\cap B_{\rho_j}(x_j)} F(\hat{n})\,d\Hausdorff^{n-1}\right|
    < \rho_j^{n-1} \eps^\alpha.
  \end{equation*}
  Summing everything together we obtain
  \begin{equation*}
    \lim_{r\to 0}\left|\frac{1}{r}\|f_r\ast \One_E\|_{L^2(\Real^n)}^2
    - \int_{\partial^* E} F(\hat{n})\,d\Hausdorff^{n-1}\right|
    \leq C\eps^\alpha(1 + \sum_j \rho_j^{n-1}).
  \end{equation*}
  Using again that $B_{\rho_j}(x_j)$ was delta acceptable,
  $\rho_j^{n-1} \leq 2P(E;B_{\rho_j}(x_j))$.  The sum on the right is therefore
  bounded by $2P(E)$.  The conclusion follows in the limit $\eps\to 0$.
\end{proof}

The rest of this section is dedicated to proving Lemma~\ref{local-lem}.  This
is done in a series of steps.  First, we consider the case that $E$ is a
half-space, and so derive a formula for $F(\nu)$.  Second, we allow $E$ to
be given by a $C^1$ surface whose normal is everywhere close to $\nu$,
and perform a computation which involves changing coordinates.
Finally, we prove an approximation lemma which allows us to modify a set
of finite perimeter on a small set so that its surface is $C^1$.

\subsection{The half-space}
In this brief subsection, we derive a formula for $F(\nu)$ so that
Lemma~\ref{local-lem} holds.

We first compute
\begin{equation*}
    f\ast \One_{H_\nu}(x) =
  \int_{y\cdot \nu\leq 0}f(x-y)dy,
\end{equation*}
which only depends on $x\cdot \nu$.  Now, using scale and
translation symmetries
\begin{align}
  \label{H-calcn}
  \lim_{r\to 0} \frac{1}{r}\|f_r\ast \One_{H_\nu}\|_{L^2(B)}^2
  &= \omega_{n-1} \int_\Real |(f\ast \One_{H_\nu})(t\nu)|^2\,dt.
\end{align}
We now work on simplifying the expression
$\int_\Real |(f\ast H_\nu)(t\nu)|^2\,dt$.  We make a first simplification,
which is to define the function $f_\nu\in C^\infty(\Real)$ which takes the
integral of $f$ along planes perpendicular to $\nu$:
\begin{equation}
  f_\nu(t) = \int_{y\cdot\nu = t} f(y)\,d\Hausdorff^{n-1}(y).
\end{equation}
We may now write $(f\ast H_\nu)(t\nu) = f_\nu\ast H(t)$, where $H$ is the
usual Heaviside function on $\Real$.  Moreover, we observe the identity
\[
  \int_\Real |\One_H\ast \phi(t)|^2 \,dt = -\frac{1}{2}
  \int_\Real\int_\Real f(t)f(s)|s-t|\,ds\,dt.
\]
This can be checked by expanding the left hand side and changing the order
of integration, or else by observing that the left hand side
is equivalent to $\int_\Real \phi \Delta^{-1}\phi$, and then using the
fact that $|x|$ is the fundamental solution to the Laplacian in one dimension.

Applying this identity to the right hand side of Equation~\eqref{H-calcn}, we
obtain
\[
    \int_\Real|(f\ast \One_{H_\nu})(t\nu)|^2\,dt =
  -\frac{1}{2} \int_\Real\int_\Real f_\nu(t)f_\nu(s)|s-t|\,dsdt.
\]
This allows us to give a reasonably explicit formula for $F(\nu)$:
\begin{equation}
    \begin{split}
  F(\nu) &= -\frac{1}{2}\int_\Real\int_\Real f_\nu(t)f_\nu(s)|s-t|\,dsdt. \\
  &= -\frac{1}{2}\int_{\Real^n}\int_{\Real^n} f(x)f(y)|(x-y)\cdot\nu|\,dsdt.
  \end{split}
  \label{Fnu-defn-eq}
\end{equation}
We may use this expression to verify that $F(\nu)$ is Lipschitz in $\nu$, as
\begin{align*}
    |F(\nu) - F(\nu')|
    &\leq |\nu-\nu'|\int_{\Real^n}\int_{\Real^n} |f(x)||f(y)||x-y|\,dxdy \\
    &\leq |\nu-\nu'|\int_{\Real^n}\int_{\Real^n} |f(x)||f(y)|(|x|+|y|)\,dxdy
    \\&\leq 2|\nu-\nu'|\|f\|_{L^1}\||x|f\|_{L^1}.
\end{align*}

\subsection{Differentiable graphs}
\label{C1-graphs-sec}
We now show that Lemma~\ref{local-lem} holds in the case that $E$ is the
graph of a $C^1$ function with small gradient.  To state the result we need a
little notation.  Given a unit vector $\nu_0\in\Sphere^{n-1}$, we use
$\pi_{\nu_0}$ to mean the orthogonal projection onto the hyperplane orthogonal
to $\nu_0$.
\begin{lemma}
  \label{C1-close}
  Let $B\subset\Real^n$ denote the unit ball, and let $G\subset \Real^n$
  be a subgraph for a function $g\in C^1(\Real^{n-1})$:
  \[
      G = \{x\in\Real^n; x\cdot\nu_0 \leq g(\pi_{\nu_0}(x)) \}.
  \]
  Suppose that $\|g\|_{C^0} < \eps$ and $\|g\|_{C^1} < \eps$.  Then
  \begin{equation}
    \label{err-est}
    \limsup_{r\to 0}
    \left|\frac{1}{r}\|f_r\ast \One_G\|_{L^2(B)}^2
    - \omega_{n-1}F(\nu_0)\right| \leq C\eps^\alpha
  \end{equation}
  for constants $C>0$ and $\alpha>0$ depending only on the kernel $f$ and
  the dimension $n$.
\end{lemma}
\begin{remark}
    According to Theorem~\ref{sofp-thm} and the Lipschitz continuity of $F$,
    the correct exponent is $\alpha=1$.  However it is more elementary to
    prove the lemma with smaller $\alpha$.
\end{remark}
\begin{proof}
    From the calculation in the previous section, it is sufficient to bound
    \[
        \frac{1}{r}\left|\|f_r\ast\One_G\|_{L^2(B)}^2
        - \|f_r\ast \One_{H_{\nu_0}}\|_{L^2(B)}^2\right|.
    \]
    As $\nu_0$ is fixed we will write $H=H_{\nu_0}$ for short.  First we will
    control the contribution to the integral from points far away from
    $\partial G$ and $\partial H$.

    Fix a point $x$ away from the boundary $\partial G$,
    and let $\ell=|x\cdot\nu_0-g(\pi_{\nu_0}(x))|$.  Because the function $g$
    is Lipschitz, there exists a constant $c$ such that the ball
    $B_{c\ell}(x)\cap \partial G = 0$.  In particular, by the cancellation
    property $\int f=0$, we have
    \begin{align*}
        f_r\ast \One_G(x)
        &\leq \int_{|x|>c\ell} |f_r(x)|\,dx \\
        &\leq \int_{|y|>c\ell r^{-1}} |f(y)|\,dy \\
        &\leq \frac{Cr}{\ell}\int |y| |f(y)|\,dy \leq \frac{Cr}{\ell}.
    \end{align*}
    Define the set $G_\beta$ of points sufficiently far from $\partial G$ by
    \[
        G_\beta := \{x\in B \mid
        |x\cdot\nu_0-g(\pi_{\nu_0}(x))|>\eps^{-\beta}r\},
    \]
    where the exponent $\beta$ will be chosen later.
    Then from the estimate above we conclude
    \[
        \frac{1}{r}\int_{G_\beta} |f_r\ast \One_G(x)|^2 \,dx
        \leq C \eps^\beta.
    \]
    Analogously, defining
    \[
        H_\beta := \{x\in B \mid |x\cdot \nu_0| > \eps^{-\beta}r\}
    \]
    we have
    \[
        \frac{1}{r}\int_{H_\beta} |f_r\ast \One_H(x)|^2\,dx \leq C\eps^\beta.
    \]

    It remains to control the contribution of the integrand near $\partial G$
    and $\partial H$.  When $r\ll \eps$, it does not suffice to use the triangle
    inequality on the difference $|f_r\ast (\One_G-\One_G)(x)|$ because there
    is not much cancellation.  To achieve cancellation we must change coordinates
    to match the regions on which $f_r\ast \One_G$ and $f_r\ast \One_H$ agree.
    To define this change of variables we write $z=(z',t)\in\Real^n$ where
    $t=z\cdot\nu_0$ and $z'=\pi_{\nu_0}(z)$.  Then the coordinate transformation
    we need is the map $\Phi(z',t)=(z',t+g(z'))$.

    The map $\Phi$ is differentiable with
    differentiable inverse, and is designed so that $\Phi(H)=G$.  Moreover
    we compute the Jacobian
    $\nabla\Phi = Id + \nu_0\otimes (\nabla(g\circ \pi_{\nu_0}))$, so in particular
    \begin{equation}
        |\nabla\Phi - Id| \leq \eps.
        \label{jacobian-close-bd}
    \end{equation}
    Now we compute, using the change of coordinates $y=\Phi(z)$,
    \begin{align*}
        f_r\ast\One_G(\Phi(x_0)) &= \int_G f_r(\Phi(x_0)-y)\,dy \\
        &= \int_G f_r(\Phi(x_0)-\Phi(z))|\det\nabla\Phi|\,dz.
    \end{align*}
    We can now compare this pointwise against $f_r\ast\One_H(x_0)$,
    \begin{align*}
        |f_r\ast \One_G(\Phi(x_0)) - f_r\ast \One_H(x_0)|
            &\leq \int_H |f_r(\Phi(x_0)-\Phi(z))|\det\nabla\Phi| - f_r(x_0-z)|\,dz \\
            &\leq \int_H |\det\nabla\Phi||f_r(\Phi(x_0)-\Phi(z)) - f_r(x_0-z)|\,dz \\
            & \quad+ \int_H ||\det\nabla\Phi|-1||f_r(x_0-z)|\,dz
            \\&= e_1 + e_2.
    \end{align*}
    To control the first error term $e_1$ we split the integral into a near and
    a far part.  For the near part, we use the fact that $|\nabla f_r|<Cr^{-1-n}$
    and the bound~\eqref{jacobian-close-bd} to deduce that
    \[
        |f_r(\Phi(x_0)-\Phi(z)) - f_r(x_0-z)| \leq |x_0-z|\eps r^{-1-n}.
    \]
    We now estimate
    \begin{align*}
        e_1 &\leq
        \eps \int_{|x_0-z|<\eps^{-a}r} |x_0-z| r^{-1-n}\,dz
        + C \int_{|x_0-z|\geq \eps^{-a}r} |f_r(x_0-z)|\,dz
        \\
        &\leq C \eps^{1-(n+1)a} + C \eps^a.
    \end{align*}
    Choosing $a=(n+2)^{-1}$ we arrive at $e_1 \leq C \eps^{1/(n+2)}$.
    Using again the Jacobian bound~\eqref{jacobian-close-bd} and the fact that
    $f_r$ is bounded in $L^1$, we can easily estimate $e_2 < C\eps$.  In
    summary
    \[
        |f_r\ast \One_G(\Phi(x_0)) - f_r\ast \One_H(x_0)| < C\eps^{\frac{1}{n+2}}.
    \]
    We can finally estimate
    \[
        \frac{1}{r}\Bigg|
        \int_{B\setminus G_\beta} |f_r\ast \One_G(x)|^2\,dx -
        \int_{B\setminus H_\beta} |f_r\ast \One_H(x)|^2\,dx \Bigg|
    \]
    by changing variables.  The main contribution comes from the integral
    over $B\setminus H_\beta \cap \Phi^{-1}(B\setminus G_\beta)$
    where the integrand is bounded by
    $C\eps^{1/(n+2)}$ and the volume of integration is $\eps^{-\beta}r$.
    Choosing now $\beta=(2n+4)^{-1}$, we have shown the lemma with exponent
    $\alpha = (2n+4)^{-1}$.
\end{proof}

Because $F(\nu)$ is Lipschitz, the normal to $\partial G$ is $\eps$-close to $\nu$,
and the map from $B'$ to
$\partial G$ nearly preserves area, we are led to the following corollary.
\begin{corollary}
  \label{C1-calc}
  Let $B$ and $G$ as in Lemma~\ref{C1-close}, and let $\nu$ denote the
  unit normal to $\partial^* G$.  Then for the same constant $\alpha$
  appearing in Lemma~\ref{C1-close},
  \[
  \limsup_{r\to 0} \left|\frac{1}{r}\|f_r\ast G\|_{L^2(B)}^2
  - \int_{\partial^*G} F(\nu) \,d\Hausdorff^{n-1}\right| \leq C\eps^\alpha.
  \]
\end{corollary}

\subsection{Sets of finite perimeter}
\label{sofp-approx}
Finally, the case of sets of finite perimeter can be resolved using an
approximation result that allows us to locally replace the reduced boundary
$\partial^*E$ by the graph of a Lipschitz function.  This can be done for
sets which are close to half-spaces in the sense of Lemma~\ref{local-lem}.

\begin{definition}
  If $E\subset B_1$, and $\delta>0$, we say that $E$ is \emph{$\delta$-close}
  to the half space $H_\nu$ if the following bounds hold:
  \begin{align*}
    |E \Delta H_\nu| &< \delta \\
    |P(E;B) - \omega_{n-1}| &< \delta.
  \end{align*}
\end{definition}

The following structural result from the theory of minimal surfaces allows us
to approximate the boundaries of sets $E$ that are $\delta$-close to the half
space $H_\nu$ by the graph of a $C^1$ function.  The theorem is analgous to the
Lipschitz truncation of Sobolev functions introduced by Acerbi and
Fusco to study lower semicontinuity of functionals appearing in calculus of
variations~\cite{acerbi1988approximation} and can be proven using similar methods.
The theorem in~\cite{maggi2012sets} is more general, and shows that if $E$ is
close to having minimal area, then one can even ask that $g$ is close to a
harmonic function.
\begin{theorem}[Consequence of {\cite[Theorem 23.7]{maggi2012sets}}]
  For any $\eps>0$, there exists a $\delta>0$ such that for all
  sets of finite perimeter $E\subset B$ which are $\delta$-close to $H_\nu$,
  there exists a function $g\in C^1(\Real^{n-1})$ whose subgraph
  \begin{equation}
    G = \{x\in B; x\cdot \nu \leq g(\pi_\nu(x))\}
    \label{subgraph-defn}
  \end{equation}
  satisfies $|G\Delta E|<\eps$, $P(G\Delta E)<\eps$, and
  $\|g\|_{C^1} < \eps$.
  \label{C1-approx}
\end{theorem}

\begin{proof}[Proof of Lemma~\ref{local-lem} using Theorem~\ref{C1-approx}]
  Let $\eps>0$, and take $\delta>0$ according to Theorem~\ref{C1-approx}.
  If $E$ is $\delta$-close to $H_\nu$, then we can apply
  Theorem~\ref{C1-approx} to obtain a $G\subset B$ with the property
  that $P(G\Delta E) < \eps$. From
  Propositions~\ref{local-limit-lem} and~\ref{continuity-lem}, we see that
  \[
      \limsup_{r\to 0}\left|
      \frac{1}{r}\|f_r\ast \One_E\|{L^2(B)}^2
      - \frac{1}{r}\|f_r\ast\One_G\|_{L^2(B)}^2\right| < C\eps.
  \]
  Moreover, since the normals of $G$ and $E$ agree except on a small set,
  \[
      \left|\int_{\partial^*E}F(\nu)\,d\Hausdorff^{n-1}
      -\int_{\partial G}F(\nu)\,d\Hausdorff^{n-1}\right| < C\eps,
  \]
  we can conclude upon using Corollary~\ref{C1-calc}.
\end{proof}

\section{Results on bounded functions of bounded variation}
\label{no-jumps-sec}
\subsection{The $H^{1/2}$ norm finds the jump set}
In this section we prove that the decay of the Fourier modes of a function
in $BV\cap L^\infty$ is tied to a quadratic integral over the jump set.
\begin{theorem}
    Let $\eta>0$ be a small parameter and $D\subset\Real^n$ be a closed set.
    Then there exists a constant $C(\eta)$ such that for for any
    $u\in BV\cap L^\infty(\Real^n)$,
    \begin{equation}
        \limsup_{r\to 0^+} \frac{1}{r} \|f_r\ast u\|_{L^2(D)}^2
        \leq \eta\|u\|_{BV(\Real^n)}\|u\|_{L^\infty(\Real^n)} +
        C(\eta) \int_{J_u\cap D} |u^+-u^-|^2 \,d\Hausdorff^{n-1}.
        \label{bound-by-jump}
    \end{equation}
    \label{jump-bound-thm}
\end{theorem}
\begin{remark}
    When $f$ is a function with compact support it is possible to refine the
    argument below to prove that
    \[
        \limsup_{r\to 0^+} \frac{1}{r} \|f_r\ast u\|_{L^2(D)}^2
        \leq C \int_{J_u\cap D} |u^+-u^-|^2 \,d\Hausdorff^{n-1}.
    \]
    However, the constant $C$ depends on the radius of the support of $f$.
    I would guess that this bound remains true when $|x|f\in L^1$, but the
    methods here do not seem to suffice to prove it.
\end{remark}
\begin{proof}
    We can assume that
    \begin{equation}
        \limsup_{r\to 0^+} \frac{1}{r} \|f_r\ast u\|_{L^2(D)}^2
        > \eta \|u\|_{BV(\Real^n)}\|u\|_{L^\infty(\Real^n)},
        \label{limsup-big}
    \end{equation}
    for otherwise the bound is trivial.  In this case we will find a set
    $A\subset D$ such that
    $\Hausdorff^{n-1}(A) \lesssim \|u\|_{BV}\|u\|_{L^\infty}^{-1}$
    and $|Du|(A) \gtrsim \|u\|_{BV}$.  In particular we will use this to
    show that
    \[
        \|u\|_{BV}\|u\|_{L^\infty} \lesssim
        \int_A |u^+-u^-|^2 \,d\Hausdorff^{n-1},
    \]
    from which the claim will follow.

    The proof is divided into three steps.  In the first step, we find
    approximations $A_k$ to the set $A$ and analyze them.  Next, the sets
    $A_k$ are used to show that the limit $A$ has the desired properties.
    Finally, the existence of $A$ is used to prove~\eqref{bound-by-jump}.
    These sets are depicted in Figure~\ref{Ak-figure}.

    \begin{figure}
        \begin{center}
            \includegraphics{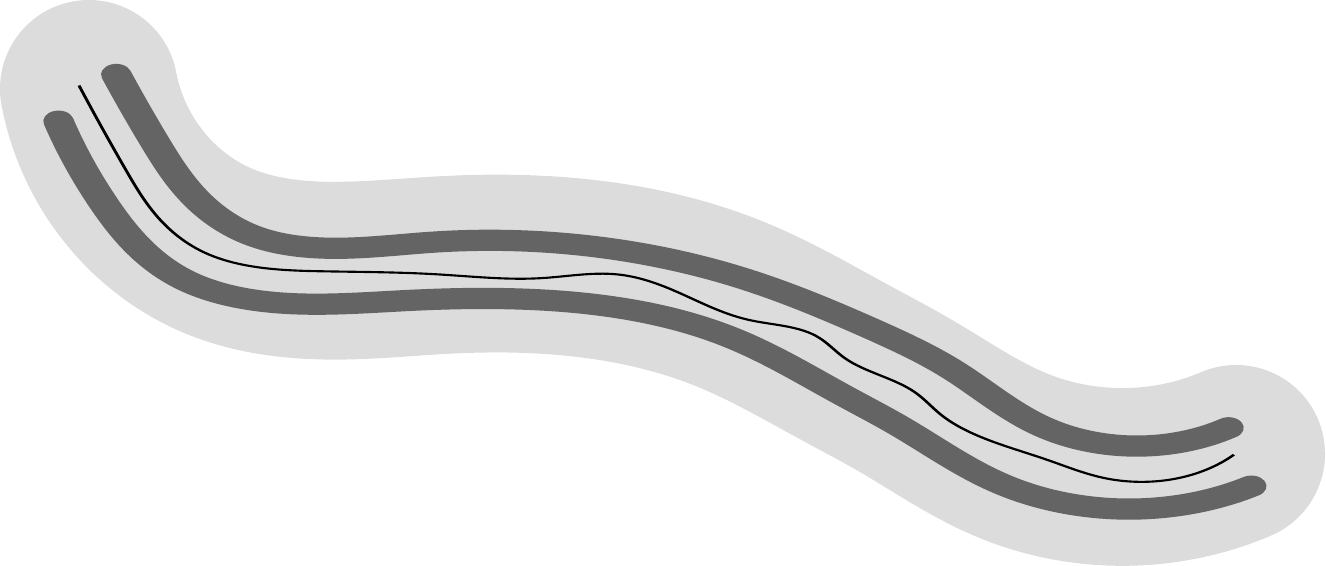}
        \end{center}
        \caption{A depiction of the construction of $A$.  The solid line
            represents a portion of $J_u$ on which $u$ has a significant jump.
            In the dark gray region $A'_k$, one has that $f_r\ast u$ is large,
            but it may be disjoint from $J_u$.
        The thicker region $A_k$ contains a portion of the jump set.}
        \label{Ak-figure}
    \end{figure}

    \vskip 12pt \noindent \emph{Step 1: Finding $A_k$}.
    With the assumption~\eqref{limsup-big} we can find a sequence $r_k\to 0$
    such that
    \[
        \limsup_{r\to 0^+} \frac{1}{r} \|f_r\ast u\|_{L^2(D)}^2
        > \frac{\eta}{2} \|u\|_{BV(\Real^n)}\|u\|_{L^\infty(\Real^n)}.
    \]
    On the other hand we have the bound
    \[
        \frac{1}{r} \|f_r\ast u\|_{L^2(D)}^2
        \leq \left(\frac{1}{r}\|f_r\ast u\|_{L^1(D)}\right)
        \left(\|f_r\ast u\|_{L^\infty(D)}\right).
    \]
    Thus, applying~\eqref{Linf-conv-leq} and~\eqref{BV-conv-leq} from
    Lemma~\ref{basic-ests} and rearranging we obtain the bounds
    \begin{align}
        \frac{1}{r_k}\|f_{r_k}\ast u\|_{L^1(D)} &\geq c_\eta \|u\|_{BV(\Real^n)}
        \label{BV-lower-bd}
        \\
        \|f_{r_k} \ast u\|_{L^1(D)} &\geq c_\eta \|u\|_{L^\infty(\Real^n)}
        \label{Linf-lower-bd}
        \\
        \|f_r\ast u\|_{L^2(D)}^2 &\geq c_\eta
        \|f_{r_k}\ast u\|_{L^1}\|f_{r_k}\ast u\|_{L^\infty}
        \label{interp-lower-bd}
    \end{align}
    Given $r_k$, define
    \[
        A_k' := \{x\in D \mid |f_{r_k}\ast u(x)| \geq \frac{1}{2}c_\eta \|u\|_{L^\infty}\}.
    \]
    The idea behind setting $A_k'$ this way is this: one should expect that
    $f_r\ast u(x)$ can only be large near large gradients of $u$.  Thus
    $A_k'$ should help us locate some of the mass of $|Du|$.  First we will
    show that in fact $A_k'$ is a set of an appreciable size.
    Indeed, using the definition of $A_k'$ we can write
    \[
        \|f_{r_k}\ast u\|_{L^2(D)}^2
        \leq \frac{1}{2}c_\eta \|u\|_{L^\infty}\|f_r\ast u\|_{L^1}
        + \int_{A_k'} |f_{r_k}\ast u(x)|^2\,dx.
    \]
    Applying the interpolation lower bound~\eqref{interp-lower-bd} on the left,
    rearranging, and bounding the integrand using the $L^\infty$ norm we
    arrive at
    \[
        \frac{1}{2}c_\eta \|f_r\ast u\|_{L^1} / \|f_r\ast u\|_{L^\infty}
        \leq |A_k'|
    \]
    Now using the $BV$ lower bound~\eqref{BV-lower-bd} and the upper
    bound~\eqref{Linf-conv-leq} this becomes
    \begin{equation}
        |A_k'| \geq \frac{r_k c_\eta^2}{2C} \|u\|_{BV} / \|u\|_{L^\infty}.
        \label{Ak-size-bd}
    \end{equation}
    This is the size one would expect if $u$ had a jump of size $\|u\|_{L^\infty}$
    along a surface of Hausdorff measure $\|u\|_{BV} / \|u\|_{L^\infty}$.  In
    this case $A_k'$ would be contained in a strip of width $r_k$ near the jump.
    To corroborate this story we would like to show that in fact the total
    variation $|Du|$ has appreciable mass near $A_k'$.  To do this we recall
    the pointwise bound~\eqref{local-conv-leq}
    \[
        |f_r\ast u(x)| \leq C r^{1-N} |Du|(B_{C(\eps)r}(x)) + \eps \|u\|_{L^\infty}.
    \]
    Setting $\eps < \frac{1}{4}$ and rearranging we have, for any $x\in A_k'$,
    \[
        \frac{c_\eta}{4}\|u\|_{L^\infty} \leq C r_k^{1-N} |Du|(B_{Cr_k}(x)).
    \]
    Integrating this over all $x\in A_k'$, using~\eqref{Ak-size-bd}
    \begin{align*}
        \frac{r_kc_\eta^3}{8C}\|u\|_{BV}
        &\leq C\int_{A_k'} r_k^{1-N} |Du|(B_{Cr_k}(x)) \,dx
        \\&\leq C r_k|Du| (A_k),
    \end{align*}
    where we have defined
    \[
        A_k = \bigcup_{x\in A_k'} B_{Cr_k}(x).
    \]
    Now relabeling constants (for we have no longer have need to be careful)
    we can summarize our main result from this step as
    \begin{equation}
        |Du|(A_k) \geq c'_\eta \|u\|_{BV}.
    \end{equation}

    \vskip 12pt \noindent \emph{Step 2: The limit $A$}.
    In this step we analyze the limsup $A$ of the sets $A_k$.  That is,
    we set
    \[
        A = \bigcap_{j=1}^\infty \bigcup_{k=j}^\infty A_k.
    \]
    By the definition of $A_k$, one has that $d(x,D) < Cr_k$ for each $x\in A_k$.
    Thus, $A\subset D$.  Moreover by Fatou's lemma we have that
    \begin{equation}
        \label{Du-bd}
        |Du|(A) \geq c'_\eta \|u\|_{BV}.
    \end{equation}
    It remains to show that $\Hausdorff^{n-1}(A)$ is small.  We will do this
    by using the sets $A_k'$ to construct efficient coverings of $A_k$.

    Indeed, let $\delta > 0$ be a radius for our covering.  For each
    $x\in A$, one has that $x\in A_k$ for infinitely many $k$.  In particular,
    for each point $x\in A$ one has that $x\in B_{Cr_m}(y)$ for some $y\in A_m'$
    and $Cr_m < \delta$.  Thus the set
    \[
        \mathcal{B}_\delta = \bigcup_{Cr_m < \delta}
        \{B_{Cr_m}(y) \mid y\in A_m'\}
    \]
    is a collection of balls of radius less than $\delta$ covering $A$.  Using
    the $5r$-covering lemma, can choose a subset
    $\{B_{Cr_i}(y_i)\}_{i=1}^N\}$ such that the balls are pairwise
    disjoint and such that $\{B_{5Cr_i}(y_i)\}_{i=1}^N\}$  covers $A$.
    Applying the pointwise bound~\eqref{local-conv-leq} to the balls in this
    cover we estimate
    \begin{align*}
        \Hausdorff^{n-1}_\delta(A)
        \leq \sum_{i=1}^N c_n (5Cr_i)^{n-1}
        &\leq \frac{C_\eta}{\|u\|_{L^\infty}} \sum_{i=1}^N |Du|(B_{Cr_i}(y_i))\\
        &\leq C_\eta \|u\|_{BV} / \|u\|_{L^\infty}.
    \end{align*}
    Since this holds for any $\delta>0$, we conclude that
    \begin{equation}
        \Hausdorff^{n-1}(A) \leq C_\eta \|u\|_{BV} / \|u\|_{L^\infty},
        \label{Amsr-bd}
    \end{equation}
    which we observe is what one would expect if $A$ were a portion of the
    jump set with nearly maximal jump magnitude.

    \vskip 12pt \noindent \emph{Step 3: Conclusion.}
    We now use the set $A$ to obtain the lower bound
    \[
        \|u\|_{BV}\|u\|_{L^\infty} \leq C(\eta)
        \int_A |u^+-u^-|^2\,d\Hausdorff^{n-1}.
    \]
    First observe that $\Hausdorff^{n-1}(A) < \infty$ implies that only
    the jump part $|u^+-u^-|\Hausdorff^{n-1}\mres J_u$ of the measure
    $|Du|$ contributes to $|Du|(A)$.
    Since $\Hausdorff^{n-1}(A)$ is small we should expect that at most points
    in $A$, the magnitude of the jump is nearly optimal.  Thus we define
    \[
        A' = \{x\in A \mid |u^+-u^-| \geq \frac{1}{2C(\eta)^2} \|u\|_{L^\infty}\}.
    \]
    Now we can bound
    \begin{align*}
        \|u\|_{BV}&\leq
        C(\eta) \int_A |u^+-u^-|\,d\Hausdorff^{n-1}
        \\&\leq \frac{1}{2}\|u\|_{L^\infty}\Hausdorff^{n-1}(A)
        + 2\|u\|_{L^\infty}\Hausdorff^{n-1}(A')
        \\&\leq \frac{1}{2}\|u\|_{BV} + 2C(\eta)\|u\|_{L^\infty} \Hausdorff^{n-1}(A').
    \end{align*}
    Rearranging we conclude that
    \[
        \Hausdorff^{n-1}(A') > c_\eta\|u\|_{BV} / \|u\|_{L^\infty}.
    \]
    Finally, we may bound
    \[
        \int_{A'} |u^+-u^-|^2\,d\Hausdorff^{n-1}
        \gtrsim  c_\eta \|u\|_{L^\infty}^2 \Hausdorff^{n-1}(A')
        \gtrsim c_\eta \|u\|_{L^\infty}\|u\|_{BV}.
    \]
\end{proof}

\subsection{The one-dimensional case.}
In this section we use Theorem~\ref{bound-by-jump} to prove an exact result
in one dimension.  We first perform a calculation in the case that $u\in BV(\Real)$
has finitely many jumps.
\begin{proposition}
    Let $u\in BV(\Real)$ be piecewise constant with jumps at finitely
    many points $x_j$, $1\leq j \leq N$.  Then
    \[
        \lim_{r\to 0} \frac{1}{r} \|f_r\ast u\|_{L^2}^2
        = c_f\sum_{j=1}^N |u^+(x_j) - u^-(x_j)|^2.
    \]
    \label{finite-jumps-lem}
\end{proposition}
\begin{proof}
    Choose $\eps < \frac{1}{2} \min_{i,j} |x_i-x_j|$ and partition $\Real$ into
    a collection of closed intervals $\{I_k\}_{k=1}^{2N+1}$ such that
    \[
        I_m = [x_m-\eps, x_m+\eps]
    \]
    for $1\leq m\leq N$.  The remaining intervals do not contain any of the
    jump points $x_j$, so $u$ is contant on such intervals.
    Thus by Proposition~\ref{local-limit-lem},
    \[
        \lim_{r\to 0}\frac{1}{r}\|f_r\ast u\|_{L^2(I_n)}^2 = 0
    \]
    for $n > N$.  On the other hand by scaling it must be that
    \[
        \lim_{r\to 0}\frac{1}{r}\|f_r\ast u\|_{L^2(I_m)}^2
        = |u^+(x_m)-u^-(x_m)|^2 \lim_{r\to 0}\frac{1}{r}\|f_r\ast H\|_{L^2((-1,1))}^2
    \]
    where $H$ is the heaviside function on $\Real$.  Taking
    \[
        c_f = \lim_{r\to 0}\frac{1}{r}\|f_r\ast H\|_{L^2((-1,1))}^2
    \]
    concludes the proof.
\end{proof}

\begin{theorem}
    Let $u\in BV(\Real)$ and $f$ be a kernel as in Lemma~\ref{basic-ests}.
    Then for some constant $c_f$,
    \[
        \lim_{r\to 0^+} \frac{1}{r} \|f_r\ast u\|_{L^2}^2
        = c_f \sum_{x \in J_u} |u^+(x) - u^-(x)|^2
    \]
    \label{one-dim-thm}
\end{theorem}
\begin{proof}
    Let $u\in BV(\Real)$.  We will decompose $u$ into a main
    part that is a piecewise continuous function with finitely many jumps,
    a continuous part that contributes nothing according to
    Theorem~\ref{jump-bound-thm}, and an error term that is small in $BV$.

    For each jump height $\delta>0$,
    define the set $J_\delta$ of jumps of height at least $\delta$:
    \[
        J_\delta = \{x\in J_u \mid |u^+(x) - u^-(x)| > \delta\}.
    \]
    Then one can decompose $Du$ as
    \[
        Du = Du\mres J_\delta + Du^c + Du\mres (J_u\setminus J_\delta).
    \]
    Integrating, we can define the decomposition
    \[
        u = u_b + u_c + u_s
    \]
    where the subscripts refer to the parts that are `big', `continuous',
    and `small'.  We have, using Proposition~\ref{finite-jumps-lem},
    Theorem~\ref{jump-bound-thm}, and Lemma~\ref{basic-ests} respectively:
    \begin{align*}
        \lim_{r\to 0}\frac{1}{r}\|f_r\ast u_b\|_{L^2}^2
        &= \sum_{J_\delta} |u^+(x)-u^-(x)|^2 \\
        \lim_{r\to 0}\frac{1}{r}\|f_r\ast u_c\|_{L^2}^2 &= 0 \\
        \limsup_{r\to 0}\frac{1}{r}\|f_r\ast u_s\|_{L^2}^2 &\leq
        C|Du|(J_u\setminus J_\delta)^2.\\
    \end{align*}
    In the last bound the square comes from the fact that
    $\|u_s\|_{L^\infty}\leq \|u_s\|_{BV}$.  The theorem follows from applying
    the continuity from Proposition~\ref{continuity-lem} and sending $\delta\to 0$.
\end{proof}

\bibliographystyle{alpha}
\bibliography{./refs}

\end{document}